\documentclass{amsart}
\usepackage{chngcntr}
\usepackage{apptools}
\usepackage{graphicx}
\usepackage{mathrsfs}
\usepackage{enumerate}
\usepackage{ upgreek }
\usepackage[latin1]{inputenc} 
\usepackage{color}
\usepackage{amsthm,amssymb,verbatim}
\usepackage{enumerate}

\newcommand{\rr}{{\mathbb R}}

\newcommand{\Sp}{{\mathbb S}}

\newcommand{\eps}{\varepsilon}
\newcommand{\set}[1]{\left\{#1\right\}}
\newcommand{\pare}[1]{\left(#1\right)}
\newcommand{\abs}[1]{\left|#1\right|}
\newcommand{\norm}[1]{\abs{\abs{#1}}}
\newcommand{\dive}[2]{\mbox{div}_{_{#1}}\pare{#2}}
\newcommand{\W}{{\mathcal W}}

\newcommand{\vv}{{\mathbf v}}
\newcommand{\ww}{{\mathbf w}}
\newcommand{\uu}{{\mathbf u}}

\newcommand{\FF}{{\mathbf F}}
\newcommand{\dist}[1]{\mbox{dist}\left({#1}\right)}

\newcommand{\pI}[1]{\left <{#1}\right >}

\newcommand\restr[2]{{
  \left.\kern-\nulldelimiterspace 
  #1 
  \vphantom{\big|} 
  \right|_{#2} 
  }}

 \newtheorem{theorem}{Theorem}[section]

\newtheorem{proposition}[theorem]{Proposition}

\newtheorem{definition}[theorem]{Definition}

\newtheorem{remark}[theorem]{Remark}

\newtheorem{claim}[theorem]{Claim}

\usepackage{hyperref}
\usepackage[alphabetic, msc-links, backrefs]{amsrefs}

\begin{document}

\title[Non-existence Theorem for $Q_{n-1}$-translators]{An example of rotationally symmetric  $Q_{n-1}$-translators and a non-existence theorem in $\rr^{n+1}$}
\author[J. Torres Santaella]{Jos\'e  Torres Santaella}
\thanks{The author was partially supported by the projects  $P20_01391$ Junta de Andalucía (I+D+i, PAIDI 2020) and MTM2016-78807-C2-1-P (MINECO-FEDER)} 
\address{Departamento de Matemáticas\\ Pontificia Universidad Católica de Chile\\ Santiago\\ Chile}
\email{jgtorre1@uc.cl}

\maketitle

\begin{abstract}
The main result in this paper is a non-existence Theorem of entire $Q_{n-1}$-translators in $\rr^{n+1}$. In addition, an example of non-entire complete $Q_{n-1}$-translator has been found and a Tangential Principle for $Q_k$-translators in  $\rr^{n+1}$.
\end{abstract}

\section{Introduction}

Translating solitons, or translators for short, in Geometrical Analysis, are self-similar solutions of curvature flows which evolve by translation over a fixed direction. These types of solutions have been extensively studied in recent decades due to their connection with singularity models in curvature flows, minimal surfaces, and physical phenomena \cite{antonio martinez_2019}. For a general review of solitons in different areas of Mathematics we refer the reader to \cite{jlauret}. 
\newline
\newline
In the case of extrinsic geometric flows, the most studied translators are those which appear in the Mean Curvature Flow.  They have been studied in \cite{ilmanen_1994} and \cite{huisken-sinestrari_1999}, as type $2$-singularities models of the Mean Curvature Flow when the initial data is mean convex, and as minimal hypersurface in $\left(\rr^{n+1},e^{\pI{p,v}}\pI{\cdot,\cdot}\right)$, respectively. 
In consequence, Mean Curvature translators can be taken as a bridge between minimal surfaces and singularity models of the Mean Curvature Flow. 
\newline
\newline
In the spirit of expanding the analysis of Mean Curvature translators to other types of extrinsic geometric flows, we consider translators of the $Q_k$-flow. 
\newline
Being more precise, a closed manifold $M^n$ evolves under the $Q_k$-flow in $\rr^{n+1}$ if, for a given immersion $F_0:M\to\rr^{n+1}$, there exist a $1$-parameter family of immersions $F:M\times[0,T]\to\rr^{n+1}$ such that
\begin{align}\label{Q_k-flow}
\begin{cases}
 \left(\dfrac{\partial F}{\partial t}\right)^{\perp}&=Q_k(\lambda),\mbox{ on }M\times(0,T),
 \\
 F(\cdot,0)&=F_0(\cdot),
 \end{cases}
 \end{align}
where $(\cdot)^{\perp}$ denotes the orthogonal projection onto the normal bundle of $TM_t$ in $T\rr^{n+1}$, $M_t=F(M,t)$, $Q_k(\lambda)=\dfrac{S_{k+1}(\lambda)}{S_k(\lambda)}$ and $S_l(\lambda)$ denotes the symmetric elemental polynomial in $n$-variables of order $l$ evaluated at the principal curvatures of $M_t$.  This flow has been studied by many authors and we refer the reader to \cite{andrews_2004}, \cite{dieter_2005}, \cite{Choi-Daskalopoulos_2016} and \cite{Yo} for some related work on this curvature flow. 
\newline
In this context, a $Q_k$-translator is a solution $F(x,t)$ to \eqref{Q_k-flow} of the form 
\begin{align}\label{ine}
F(x,t)=F_0(x)+tv,
\end{align}
where $v\in\Sp^n$ is the direction of translation. Moreover, a $Q_k$-translator can be seen as a hypersurface which satisfies the equation
\begin{align}\label{Q_k-trans}
Q_k(\lambda)=\pI{\nu,v},	
\end{align}	
where $\nu$ and $\lambda$ are  the normal unit vector and the principal curvature vector of $M_0=F_0(M)$ in $\rr^{n+1}$, respectively. 
\newline
From the PDE perspective, equations \eqref{Q_k-flow} and \eqref{Q_k-trans} are fully nonlinear for $k>0$ and quasilinear for the case $k=0$. Additionally, it is a well known result that the function $Q_k$ is 1-homogeneous, strictly increasing in each coordinate, and concave when the principal curvatures of $M_0$ belong to the cone 
\begin{align}\label{cone}
\Gamma_{k+1}:=\set{\lambda\in\rr^n:S_l(\lambda)>0,\mbox{ for }l=0,\ldots,k+1}.
\end{align}
We refer to the reader to \cite{caffarelli_nirenberg_spruck_1988} and \cite{andrews_2004} for a proof of these facts. 
\newline

In this paper, we focus on similitudes and differences of $Q_k$-translators for cases $k>0$ and $k=0$. 
\newline
An important similarity is that there are non closed $Q_k$-translators in $\rr^{n+1}$ such that their principal curvatures are in $\Gamma_{k+1}$. 
Indeed, if $M_0$ is a  $Q_k$-translator of this kind, then $M_0$ is a graph over a hyperplane orthogonal to $v$.
\newline
Let $u=\pI{F_0,v}$ be the height function of $M_0$. Then, after choosing normal coordinates at $p\in M$, we have the following equations
\begin{align}\label{i1}
	\nabla_i u=\pI{e_i,v}\mbox{ and }\nabla_j\nabla_iu=h_{ij}\pI{\nu,v}.
\end{align}
Recall that in normal coordinates we have $(\nabla_je_i)^{\top}=0$ and $h^i_j=h_{ij}=\lambda_i\delta^i_j$ at $p$, where $\lambda_i$ denote the principal curvatures of $M_0$ and $\delta_j^i$ is the Kronecker's delta function. Consequently, by multiplying Equation \eqref{i1} by $\frac{\partial Q_k}{\partial h_{ij}}$, we obtain
\begin{align*}
	\square_ku =Q_k\pI{\nu,v}=\pI{\nu,v}^2=1-\abs{v^{\top}}^2=1-|\nabla u|^2,
\end{align*}
where $\square_kf=\dfrac{\partial Q_k}{\partial h^i_j}\nabla^i\nabla_jf$. Here we are using the Einstein summation convention over the indexes. 
\newline
 Finally, since $M_0$ is closed, $u$ reaches an interior maximum. However, the Maximum Principle implies that $u$ is constant. Therefore, $M_0$ is a piece of a hyperplane which contradicts the compactness of $M_0$.
\newline

In addition, we do not expect a correspondence between $Q_k$-translators and minimal hypersurfaces in weighted Euclidean spaces for $k>0$. For instance, if we consider the functional given by, 
\begin{align*}
A_k(M)=\int_MS_k(\lambda)e^{(k+1)\pI{p,v}}dA,
\end{align*}
where $M$ is a closed hypersurface in $\rr^{n+1}$. Then, the first variation of $A_k$ is given by the formula
\begin{align}\label{I2}
			\restr{\dfrac{d}{dt}A_k(M)}{t=0}=\int_M\left(f(k+1)\left(S_{k}\pI{\nu,v}-S_{k+1}\right)+\dive{M}{T_{k-1}\nabla f}\right)e^{(k+1)\pI{p,v}}dA,
		\end{align}		
for $f\in\mathcal{C}^{\infty}(M)$. A proof of Equation \eqref{I2} can be easily derived from the variations of the functionals calculated in \cite{reilly_1973}. Here, the operator $T_k=T_k(\W)$ is the $k$th-Newton transformation of the the Weingarten map of $M$, that is
		\begin{align*}
			T_k&=S_{k}I-S_{k-1}\W+\ldots+(-1)^{k-1}H\W^{k-1}+(-1)^{k}\W^k
			\\
			&=S_{k}I-\W T_{k-1}. 
		\end{align*}
We note that the term $\dive{M}{T_{k-1}\nabla f}$ only vanishes when $k=0$. In particular, this implies that $Q_k$-translators are critical points of $A_k$ only for $k=0$. 
\newline

The main difference covered in this paper is that entire $Q_{n-1}$-translators do not exist in $\rr^{n+1}$, in contrast with the bowl soliton of the Mean Curvature case  found  in \cite{haslhofer_2015}. 
\newline
We prove this result by first showing the existence of a complete strictly convex non-entire $Q_{n-1}$-translator for $n\geq 2$, and we conclude with a comparison argument. 
To be more precise, we find an explicit solution to Equation \eqref{Q_k-trans}  for case  $n=2$ (see Figure \ref{fig}) and generalize it to higher dimensions. Our result reads as follows.
\begin{theorem}\label{T1}
	For each $n> 2$  there exists a complete, strictly convex $Q_{n-1}$-translator in $\rr^{n+1}$ of the form
	\begin{align*}
		\set{(x,u(|x|))\in\rr^{n+1}:\:0\leq |x|<\frac{1}{n}},
	\end{align*}	
such that $u\to\infty$ when $|x|\to n^{-1}$.
\end{theorem} 

\begin{figure} \label{fig}
	\centering
	\includegraphics[width=0.33\columnwidth]{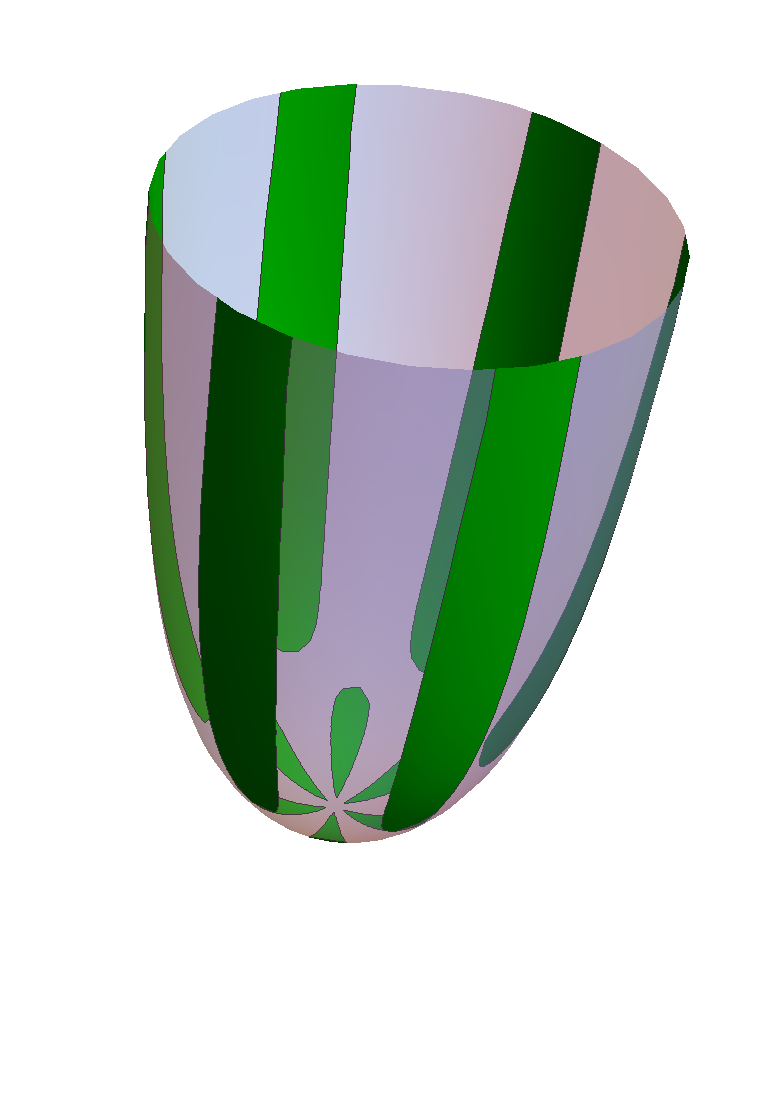}
	\vskip-10mm
	\caption{ $Q_1$-translator $z=-\ln(1-x^2-y^2)$ in $\rr^3$.}
\end{figure}
Then, we derive an interior and a boundary tangential principle for $Q_k$-translators.
\begin{theorem}\label{T2}
	Let $k>0$ and $M_1,M_2\subset\rr^{n+1}$ be two embedded connected $Q_k$-translators in the $x_{n+1}$-direction such that
	\begin{enumerate}
		\item $M_1$ is strictly convex.
		\item The principal curvatures of $M_2$ are in the cone $\Gamma_{k+1}\cap\set{\lambda\geq 0}$.
\end{enumerate}
Then,
\begin{enumerate}
	\item[(a)] (\textbf{Interior principle}) Assume that there exists an interior point $p\in M_1\cap M_2$ such that the tangent spaces coincide at $p$. If $M_1$ lies at one side of $M_2$, then both hypersurfaces coincide.
	
	\item[(b)] (\textbf{Boundary principle}) Assume that the boundaries $\partial M_i$ lie in the same hyperplane $\Pi$ and the intersection of $M_i$ with $\Pi$ is transversal. If $M_1$ lies at one side of $M_2$ and there exist $p\in \partial M_1\cap \partial M_2$ such that the tangent spaces coincide, then both hypersurfaces coincide. 
\end{enumerate}
\end{theorem}
Finally, by combining the above theorems, we prove by standard arguments the main results.   
\begin{theorem}\label{T3}
	There are no entire $Q_{n-1}$-translators in $\rr^{n+1}$ such that their principal curvatures belong to $\Gamma_{n}$. 
\end{theorem}

Moreover, by the Methods of Moving Planes of Alexandrov, we show that the solution found in Theorem \ref{T1} is unique among any other strictly convex $Q_{n-1}$-translator asymptotic to the cylinder $\Sp^{n-1}(r_n)\times\rr$.   
\begin{theorem}\label{T4}
	The only $Q_{n-1}$-translators in $\rr^{n+1}$ such that their principal curvatures belong to $\Gamma_{n}$ and is asymptotic to the cylinder $\Sp^{n-1}(r_n)\times\rr$ is the  one found in Theorem \ref{T1}. Here $r_n=1$ if $n=2$ and $n^{-1}$ if $n>2$.  
\end{theorem}
The structure of the paper is summarized as follows: In Section \ref{sec2} we state the equations for rotationally symmetric $Q_k$-translators that we use along this paper. In Section \ref{sec3} we show an explicit $Q_{1}$-translator in $\rr^3$ and  prove Theorem \ref{T1}. In Section \ref{sec 4} we prove Theorem \ref{T2} and  Theorem \ref{T3}. Finally, in section \ref{sec 5} we prove Theorem \ref{T4}.
\newline
\newline
\textbf{Acknowledgment:} The author would like to thank F. Martín and M. S\'aez for bringing this problem to his attention and for all the support they have provided. Furthermore, I would like to thank A. Martinez-Triviño for his advice and encouragement about this topic.

\section{Graphical Rotationally Symmetric $Q_k$-Translators}\label{sec2}

In this section, we derive some geometric equations for graphical rotationally symmetric $Q_k$-translators in $\rr^{n+1}$. 
\\
 Let $M$  be a $Q_k$-translator. Then, after applying  an isometry of $\rr^n$, we may locally describe $M$ as the graph of a function $u:B_r(0)\subset\rr^{n}\to\rr$. Then, we can write Equation \eqref{Q_k-trans} as
\begin{align}\label{graph}
    \dfrac{1}{\sqrt{1+|D u|^2}}=Q_k(D u,D^2u),
\end{align}
where $Q_k(q,X)=\dfrac{S_{k+1}(q,X)}{S_k(q,X)}$, and the function $S_{k}(q,X)$ is given below for $q\in\rr^n$ and a symmetric matrix $X$.
\begin{proposition}
Let $M$ be a hypersurface in $\rr^{n+1}$ and let $p\in M$. If around $p$ the hypersurface is described by the graph of a function $u:B_r(0)\subset\rr^n\to\rr$ it follows that the elementary symmetric polynomial on the principal curvatures of $M$ at $p$ are given by,
\begin{align*}
    S_{k}(D u,D^2 u):=S_{k}(\W)=\dfrac{\delta_{i_1\ldots i_k}^{l_1\ldots l_k}}{k} h_{i_1}^{l_1}\ldots h_{i_k}^{l_k},
 \end{align*}
where
\begin{align*}
    &\delta_{i_1\ldots i_k}^{l_1\ldots l_k}=
    \begin{cases}
    &1,\mbox{ if all } i_1,\ldots,i_k\mbox{ are distinct and is an even permutation of }l_1,\ldots,l_k
    \\
    &-1,\mbox{ if all } i_1,\ldots,i_k\mbox{ are distinct and is an odd permutation of }l_1,\ldots,l_k
    \\
    &0,\mbox{ in other case }
    \end{cases}
\end{align*}
and $h_i^l=\left(\dfrac{\delta^{lk} D_{ki}u}{\sqrt{1+|D u|^2}}-\dfrac{D^luD^ku D_{ki}u}{(1+|D u|^2)^{\frac{3}{2}}}\right)$ are the coefficients of the Weirgarten map in the standard coordinates. Here we are using the Riemann summation convention of upper and lower indexes.
\end{proposition}

\begin{proof}
The proof easily follows from Equation (1) in \cite{reilly_1973}. 
\end{proof}
In the rotationally symmetric setting, we consider cylindrical coordinates $(\theta,r,s)$ on $\rr^{n+1}$ given by, 
\begin{align*}
\theta=(\theta_1,\ldots,\theta_{n-1}),\: r=|x|\mbox{ and }s=x_{n+1},
\end{align*}
where $\theta_i$ are the canonical coordinates on $\Sp^{n-1}$ and $|\cdot|$ denotes the Euclidean norm.
\newline
Then, a rotationally symmetric $Q_k$-translator graphs, or \textbf{ros} $Q_k$-translator for short,
is a hypersurface $M$ given by the graph of a rotationally symmetric function $u:B_R(0)\subset\rr^n\to\rr$ such that,
\begin{align}\label{M}
M=\set{(r\theta,\uu(r)):\theta\in \Sp^{n-1},0\leq r< R},
\end{align} 
where $\uu(r)=u(|x|)$.
\begin{remark}
We will use the dot notation to denote radial derivatives $\frac{\partial}{\partial r}$ and also bold letters for rotationally symmetric functions.
\end{remark}

The following proposition contains geometrical properties of \textbf{ros} $Q_k$-translators.

\begin{proposition}\label{23}
Let $M$ be  a $Q_k$-translator of the form \eqref{M}, then we have  the following quantities at $p\in M$ in the cylindrical frame $\set{\partial_{\theta_1},\ldots,\partial_{\theta_{n-1}},\partial_r,\partial_s}$ :
\begin{enumerate}
    \item[(a)] Unit normal vector: $\nu=\dfrac{\partial_s-\dot{\uu}\partial_r}{\sqrt{1+\dot{\uu}^2}}$.
    
    \item[(b)] Principal directions and curvatures: $\partial_{\theta_i}$, $\dot{\uu}\partial_s+\partial_r$ and $\lambda_i=\dfrac{\dot{\uu}}{r\sqrt{1+\dot{\uu}^2}}$,  $\lambda_{n}=\dfrac{\ddot{\uu}}{(1+\dot{\uu}^2)^{\frac{3}{2}}}$, respectively.
    
    \item[(c)] The function $S_{l}$ evaluated at the principal curvatures:
    \begin{align*}
    S_{l}(\lambda)=\dfrac{\binom{n-1}{l-1}}{r^{n-1}}\dfrac{d}{dr}\left(r^{n-l}\varphi(\dot{\uu})^{l}\right),\mbox{ where }\varphi(x)=\dfrac{x}{\sqrt{1+x^2}}.   
    \end{align*}
         
    \item[(d)] The function $Q_k$ evaluated at the principal curvatures:
    \begin{align*}
        Q_k(\lambda)=\dfrac{(n-k)}{(k+1)r\sqrt{1+\dot{\uu}^2}}\cdot\dfrac{ (n-k-1)(1+\dot{\uu}^2)\dot{\uu}^2+(k+1)r\ddot{\uu}\dot{\uu}}{(n-k)(1+\dot{\uu}^2)\dot{\uu}+kr\ddot{\uu}}.
    \end{align*}
    \item[(e)] The ODE for \textbf{ros} $Q_k$-translators is given by
    \begin{align}\label{Q_k-equation}
         \ddot{\uu}=\frac{n-k}{k+1}(1+\dot{\uu}^2)\frac{\dot{\uu}}{r}\dfrac{(r(k+1)-(n-k-1)\dot{\uu}) }{((n-k)\dot{\uu}-rk)}.
    \end{align}
\end{enumerate}
\end{proposition}

\begin{proof}
We refer the reader to \cite{Torres} for a proof of the three first properties. For (d) we have,
\begin{align*}
Q_k(\lambda)&=\dfrac{S_{k+1}(\lambda)}{S_k(\lambda)}=\dfrac{(n-k)}{(k+1)}\cdot\dfrac{ \dfrac{d}{dr}\left(r^{n-k-1}\varphi(\dot{\uu})^{k+1}\right) }{ \dfrac{d}{dr}\left(r^{n-k}\varphi(\dot{\uu})^{k}\right) } 
\\
&=\dfrac{(n-k)}{(k+1)r\sqrt{1+\dot{\uu}^2}}\cdot\dfrac{ (n-k-1)(1+\dot{\uu}^2)\dot{\uu}^2+(k+1)r\ddot{\uu}\dot{\uu}}{(n-k)(1+\dot{\uu}^2)\dot{\uu}+kr\ddot{\uu}}.
\end{align*}
For  (e), we  have $\pI{\nu,\eps_{n+1}}=\dfrac{1}{\sqrt{1+\dot{\uu}^2}}$. Furthermore, Equation \eqref{Q_k-trans} can be written as
\begin{align*}
    \dfrac{1}{\sqrt{1+\dot{\uu}^2}}=\dfrac{(n-k)}{(k+1)r\sqrt{1+\dot{\uu}^2}}\dfrac{ (n-k-1)(1+\dot{\uu}^2)\dot{\uu}^2+(k+1)r\ddot{\uu}\dot{\uu}}{(n-k)(1+\dot{\uu}^2)\dot{\uu}+kr\ddot{\uu}}.
\end{align*}
Finally, Equation \eqref{Q_k-equation} follows by reordering the terms of the last equation.
\end{proof}

We finish this section by noting that Equation \eqref{Q_k-equation} can be reduced to a first order ODE by taking $\vv=\dot{\uu}$. This method was used to study the behavior at infinity in the Mean Curvature case  in \cite{clutterbuck_2005}.  
	
\begin{remark}
	For each $k\in\set{0,\ldots,n-1}$, the ODE that represents \textbf{ros}  $Q_k$-translator graphs  in $\rr^{n+1}$ is given by 
	\begin{align}
		\begin{cases}
			\dot{\vv}(r)=&\FF_k(r,\vv),
			\\
			\vv(0)=&0,
		\end{cases}	
		\end{align} 
	where the slope function $\FF_k:\set{(r,s)\in\rr^+\times\rr: s\neq\frac{k}{n-k}r}\to\rr$ is
	\begin{align*}
		\FF_k(r,s)=\frac{n-k}{k+1}(1+s^2)\dfrac{s}{r}\left(\dfrac{r(k+1) -(n-k-1)s}{(n-k)s-kr}\right).
	\end{align*}	
\end{remark}

\section{Non-trivial $Q_{n-1}$-translators}\label{sec3}

In this section we focus on the existence of \textbf{ros} $Q_{n-1}$-translators in $\rr^{n+1}$. 
\newline

The ODE for \textbf{ros} $Q_{n-1}$-translators is given by,
\begin{align}\label{ODE 1}
\begin{cases}
\dot{\vv}(r)&=\FF_{n-1}(r,\vv)=\dfrac{\vv(1+\vv^2)}{\vv-(n-1)r},
\\
\vv(0)&=0.
\end{cases}
\end{align}

\begin{remark}\label{4.1}
	As we mentioned in the introduction, Equation \eqref{ODE 1} has an explicit solution for $n=2$ of the form,
	 \begin{align*}
	 	\vv(r)=\dfrac{2r}{1-r^2},\: r\in[0,1).
	 \end{align*}	
	 Then, we can construct a $Q_1$-translator $M$ in $\rr^3$ of the form \eqref{M} by integration. Indeed, the function $\uu=\int_0^r\vv(s)ds=-\ln(1-r^2)$ satisfies, 
	\begin{align}\label{6.4}\
		\begin{cases}
			\ddot{\uu}=\dot{\uu}\dfrac{(1+\dot{\uu}^2)}{(\dot{\uu}-r)},
			\\
			\uu(0)=\dot{\uu}(0)=0.
		\end{cases}
	\end{align}
	On the other hand, we note that $M$ is complete, but non-entire since $\lim\limits_{r\to 1}\uu(r)=\infty$, and it is strictly convex because its principal curvatures
	\begin{align*}
		\lambda_1=\dfrac{2}{(1-r^2)\sqrt{1+\dot{\uu}^2}}\mbox{ and }  \lambda_{2}=\dfrac{2(1+r^2)}{(1-r^2)^2(1+\dot{\uu}^2)^{\frac32}},
	\end{align*}
 	 are positive in $[0,1)$.
\end{remark}
We summarize the result above in the following proposition,
\begin{proposition}
	The surface
	\begin{align*}
	M=\set{(r\theta,-\ln(1-r^2)):\theta\in \Sp^{1},r\in[0,1) }
	\end{align*}
   is a complete, non-entire and strictly convex \textbf{ros} $Q_1$-translator in $\rr^3$.
\end{proposition}

 \begin{remark}
 	By a comparison argument any solution $\vv(r)$ to Equation \eqref{ODE 1} is increasing. In fact, the function $\ww_0(r)=nr$ satisfies,
 	\begin{align*}
 	\dot{\ww}_0(r)=n< n(1+n^2r^2)=\dfrac{\ww_0(1+\ww_0^2)}{\ww_0-(n-1)r},\mbox{ for }r>0. 
 	\end{align*}	
 	This implies that $\ww_0$ is a sub-solution to equation \eqref{ODE 1}. Consequently, $\vv\geq \ww_0$ for $r\geq 0$, which implies $\dot{\vv}(r)=\FF_{n-1}(r,\vv)>0$ for $r>0$.
 \end{remark}
 The following proposition provides better barrier solutions in our setting.
\begin{proposition}\label{Barriers}
The functions $\ww_1(r)=\dfrac{nr}{\sqrt{1-n^2r^2}}$ and $\ww_2(r)=-\ln(1-nr)$ satisfies,
\begin{align*}
\dot{\ww}_1(r)\geq \FF_{n-1}(r,\ww_1)\mbox{ and }\ddot{\ww}_2(r)\leq \FF_{n-1}(r,\dot{\ww}_2),\mbox{ in }\left(0,\dfrac{1}{n}\right).
\end{align*}
\end{proposition}

\begin{proof}
For the function $\ww_1(r)$ we have,
	\begin{align}\label{1}
		\dot{\ww}_1=\dfrac{\ww_1(1+\ww_1^2)}{r}\geq \dfrac{\ww_1(1+\ww_1^2)}{\ww_1-(n-1)r}=\FF_{n-1}(r,\ww_1).
	\end{align}
In the inequality part of \eqref{1} we use that $\ww_1(r)\geq nr$, which holds for $r\in[0,\frac{1}{n})$.
	\newline
On the other hand, for the function $\ww_2(r)$ we have that
\begin{align*}
\ddot{\ww}_2\leq \dfrac{\dot{\ww}_2(1+\dot{\ww}_2^2)}{\dot{\ww}_2-(n-1)r}, 
\end{align*}
 is equivalent to,
\begin{align*}
n\leq\dfrac{(1-nr)^2+n^2}{n-(n-1)r+n(n-1)r^2}.
\end{align*}
An easy computation reveals that the last inequality holds for $r\in[0,\frac{1}{n})$. 
\end{proof}
\begin{remark}\label{4.4}
	Note that the functions $\int\limits_0^{r}\ww_1(s)ds$ and $\ww_2(r)$  act as  super and sub solutions, respectively, to Equation \eqref{Q_k-equation}. Then, by a comparison argument, any solution $\uu(r)$ to Equation \eqref{Q_k-equation} with initial condition $\uu(0)=\dot{\uu}(0)=0$, will exist in $[0,n^{-1})$ and develop a singularity at $r=n^{-1}$. 
\end{remark}
Now, we study the existence problem for dimension $n>2$. In the following, we are concerned with finding an interval of the form $[0,\delta_T]$ to apply standard existence arguments to Equation \eqref{ODE 1}. 
\newline
Let $0<\delta_T<\frac{\pi}{2n}$ and  we consider the Banach space
\begin{align}\label{X}
	X(\delta_T)=\set{\vv\in\mathcal{C}([0,\delta_T])
		:\vv(0)=0,\:\sqrt{r}\leq \vv(r)\leq \tan(nr)},
\end{align}
endowed with the uniform norm. Over $X(\delta_T)$ we consider the map 
\begin{align*}
	T:X(\delta_T)&\to \mathcal{C}([0,\delta_T])
	\\
	\vv(r)&\to T(\vv)(r)=\int\limits_0^{r}\dfrac{\vv(1+\vv^2)}{(\vv-(n-1)s)}ds.
\end{align*}	
The goal of the next proposition is to find $\delta_T$ such that the map $T$ is contractive.
\begin{proposition} \label{36}
The map $T$ satisfies:
\begin{enumerate}
	\item For $\delta_T\leq \min\set{n^{-1}\arccos\left(\sqrt{\frac{n}{n+1}}\right),\frac{1}{n^2}}$,  $T(X(\delta_T))\subset X(\delta_T)$ holds.
	\newline
	\item There exists $\delta_T>0$ such that $T$ is a contraction map on $X(\delta_T)$.
\end{enumerate}	
\end{proposition}
 
\begin{proof} 
 First, we note that by definition of $T$, $T(\vv)(0)=0$. Furthermore, since $\sqrt{r}\geq nr$ in $[0,\delta_T]$, it follows that $\vv(r)\geq nr$. Consequently, we have that 
 \begin{align}\label{ine}
 	\dfrac{\vv}{\vv-(n-1)r}\leq n.
 \end{align}
 Then, Equation \eqref{ine} implies $T(\vv)(r)\leq \tan(nr)$. In fact, it follows that
 \begin{align*}
 T(\vv)(r)-\tan(nr)=\int_0^{r}\left(\frac{\vv(1+\vv^2)}{(\vv-(n-1)s)}-n\sec^2(ns)\right)ds\leq 0,
\end{align*} 
where the last inequality we use $1+\vv^2(r)\leq 1+\tan^2(nr)=\sec^2(nr)$.
\newline

On the other hand, we note that $T(\vv)(r)\geq  \sqrt{r}$ is equivalent to
\begin{align}\label{ine2}
T(\vv)(r)-\sqrt{r}=\int_0^r\left(\dfrac{\vv(1+\vv^2)}{(\vv-(n-1)s)}-\dfrac{1}{2\sqrt{s}}\right)ds\geq 0.
\end{align}
Since $\sqrt{r}\leq \vv(r)\leq \tan(nr)$ on $[0,\delta_T]$, we have
\begin{align*}
\dfrac{\vv(1+\vv^2)}{(\vv-(n-1)r)}-\dfrac{1}{2\sqrt{r}}	\geq\dfrac{( (n+1)r+2r^2 -\vv)}{2\sqrt{r}(\vv-(n-1)r)}\geq 	\dfrac{(n+1)r-\tan(nr)}{2\sqrt{r}(\vv-(n-1)r)}. 
\end{align*}	
We note that the function $(n+1)r- \tan(nr)$ is non-negative on $\left[0,n^{-1}\arccos\left(\sqrt{\frac{n}{n+1}}\right)\right]$, which implies \eqref{ine2}. This finishes the first part of Proposition \ref{36}.
\newline

In what follows we are going to prove that there is $\delta_T<\min\set{n^{-1}\arccos\left(\sqrt{\frac{n}{n+1}}\right),\frac{1}{n^2}}$ such that $T$ is a contractive map on $X(\delta_T)$. Recall that a contractive map satisfies,
\begin{align*}
\norm{T(\vv_1)-T(\vv_2)}_{\infty}\leq C\norm{\vv_1-\vv_2}_{\infty},
\end{align*} 
for $\vv_1,\vv_2\in X$ and some constant $C<1$. 
\newline
First, we note that
\begin{align}\label{ine3}
\vv_i-(n-1)r\geq \dfrac{\sqrt{r}}{n}, \mbox{on }[0,\delta_T]. 
\end{align}
Then it follows that
\begin{align*}
&\abs{T(\vv_1)-T(\vv_2)}(r)=\int_0^r\abs{\dfrac{\vv_1\left(\vv_2-(n-1)s\right)(1+\vv_1^2)-\vv_2\left(\vv_1-(n-1)s\right)(1+\vv_2^2)}{(\vv_1-(n-1)s)(\vv_2-(n-1)s)}}ds
\\ 
&\leq\norm{\vv_1-\vv_2}_{\infty}n^2\int_0^{\delta_T}\frac{1}{s}\pare{(n-1)s+\vv_1\vv_2(\vv_1+\vv_2)+(n-1)s(\vv_1^2+\vv_1\vv_2+\vv_2^2)}ds 
\end{align*}
\begin{align}\label{int}
\leq\norm{\vv_1-\vv_2}_{\infty}\left(\int_0^{\delta_T}(n-1)+2\frac{\tan^3(ns)}{s}+3(n-1)\tan^2(ns)ds\right).
\end{align}
In the first inequality  we used Equation \eqref{ine3} and for the second inequality we used  $\vv_i(r)\leq\tan(nr)$. 
\newline
Finally, we note that the terms inside the integral in \eqref{int} are all bounded on $[0,\delta_T]$. Therefore, by continuity there exists some $\delta_T>0$ such that,
\begin{align*}
\delta_T\left( (n-1)+2\norm{\frac{\tan^3(nr)}{r}}_{\infty}+3(n-1)\norm{\tan^2(nr)}_{\infty}\right)<1,
\end{align*}
which finishes the proof. 
\end{proof}

\begin{proposition}\label{37}
	In the context of the above proposition, there is a unique fixed point of $T:X(\delta_T)\to X(\delta_T)$ in $\mathcal{C}^1([0,\delta_T])$. 
\end{proposition}

\begin{proof}
First, we note that by Proposition \ref{36} the map $T:X(\delta_T)\to X(\delta_T)$ is contractive. Then, the Banach Fixed Point Theorem gives a unique $\vv(r)\in X(\delta_T)$ such that $T(\vv)(r)=\vv(r)$ on $[0,\delta_T]$. 
\newline
In addition, since $T$ is an integral operator, the fixed point $\vv(r)$ is differentiable in $[0,\delta_T]$. Moreover, since $nr\leq\vv(r)\leq \tan(nr)$, we have
\begin{align}\label{bound}
\dot{\vv}(r)=\dfrac{\vv(1+\vv^2)}{\vv-(n-1)r}\leq \dfrac{\sec(nr)^2\tan(nr)}{r},
\end{align}
which is uniformly continuous on $[0,\delta_T]$. Therefore, $\vv(r)\in\mathcal{C}^1([0,\delta_T])$ since $\norm{\dot{\vv}}_{\infty}$ is dominated by the uniform norm of \eqref{bound}. 	 
\end{proof}

\begin{proof}[Proof of Theorem \ref{T1}]
Let $\vv(r)\in X(\delta_T)\cap \mathcal{C}^1([0,\delta_T])$ be the unique fixed point of $T$ found in Proposition \ref{37}. Then, we can find a solution to 
\begin{align}\label{ODE 2}
\begin{cases}
 \ddot{\uu}(r)&=\dfrac{\dot{\uu}(1+\dot{\uu}^2)}{(\dot{\uu}-(n-1)r)}
 \\
 \uu(0)&=0,\dot{\uu}(0)=0.
\end{cases}
\end{align}
 of the form $\uu(r)=\int\limits_0^r\vv(s)ds$ for $r\in[0,\delta_T]$. 
 \\
 On the other hand, as we mentioned in Remark \ref{4.4}, the solution $\uu(r)$ can be extended to the interval $[0,n^{-1})$ and it develops a singularity at $r=n^{-1}$.
 \newline
 Then, it follows that the hypersurface $M$ of the form  \eqref{M} is a complete $\textbf{ros}$ $Q_{n-1}$-translator. Moreover, the strict convexity is guaranteed by the equations for the principal curvatures given in Proposition \ref{23} together with $\ddot{\uu}>0$ on $[0,n^{-1})$.  
\end{proof}

\section{Proofs of Theorems \ref{T2} and \ref{T3}}\label{sec 4}
The proof of Theorem \ref{T2} is inspired by the given in \cite{moller_2014}  for  the case $k=0$. 
\begin{proof}[Proof of Theorem \ref{T2}]
Let $p\in M_1\cap M_2$ be an interior point such that $T_pM_1=T_pM_2$ and $M_1$ lies at one side of $M_2$. Then, after a rotation and translation, there exists $r>0$ such that each $M_i\cap B(p,r)$ is the graph of a smooth function $u_i:B_r(0)\subset\set{x_{n+1}=0}\to\rr$, where $p=(0,u_i(0))$. In addition,  since $M_1$ lies at one side of $M_2$, we assume that $u_1>u_2$ in $B_r(0)\setminus\set{0}$ and $u_1(0)=u_2(0)$.
\\
Now we define a $1$-parameter family of functions, 
\begin{align*}
u_{s}=(1-s)u_1+su_2,\mbox{ for }s\in[0,1].	
\end{align*}	
Note that for each $s\in (0,1)$, the graph of $u_s:B_r(0)\to\rr$ is strictly convex. This holds since the convex combination of a positive definite matrix with a positive semi-definite matrix is positive definite. Here, the involved matrices are the second fundamental forms of $u_1$ and $u_2$, respectively. 
\newline
In particular, this fact implies that $u_s$ is an admissible family for the operator,  
\begin{align*}
    E(s)=Q_{k}(h^i_j(s))-\pI{\nu(s),\eps_{n+1}},
\end{align*}
where the $s$ dependence is related to each $u_s$.
\\
Consequently,  we have that,
\begin{align*}
 0=E(u_1)-E(u_2)=\int\limits_0^1\dfrac{\partial}{\partial s}E(u_s)ds.
\end{align*}
Furthermore, we explicitly calculate $\dfrac{\partial E}{\partial s}(s)$. Indeed, we have
\begin{align*}
&\dfrac{\partial E}{\partial s}(s)
\\
&=\dfrac{\partial Q_k}{\partial h^i_j}\dfrac{\partial}{\partial s}h^i_{j}(s)-\dfrac{\partial}{\partial s}\dfrac{1}{\sqrt{1+|Du_s|^2}}
\\
&=\dfrac{\partial Q_k}{\partial h^i_j}\dfrac{\partial}{\partial s}\left[\left(\dfrac{\delta_{ik}}{\sqrt{1+|Du_s|^2}}-\dfrac{D_iu_sD_ku_s}{(1+|Du_s|^2)^{3/2}}\right)D_{kj}u_s\right]+
\dfrac{\pI{D(u_2-u_1),Du_s}}{(1+|Du_s|^2)^\frac{3}{2}}
\\
&=\dfrac{\partial Q_k}{\partial h^i_j}\left[\left(-\delta_{ik}\dfrac{\pI{D(u_2-u_1),Du_s}}{(1+|Du_s|^2)^{\frac{3}{2}}}-\dfrac{(D_i(u_2-u_1))D_ku_s+(D_k(u_2-u_1))D_iu_s}{(1+|Du_s|^2)^{\frac{3}{2}}}\right.\right.
\\
&\left.\left.+3\dfrac{D_iu_sD_ku_s\pI{D(u_2-u_1),Du_s}}{(1+|Du_s|^2)^{\frac{5}{2}}} \right)D_{kj}u_s-\left(\delta_{ik}+\dfrac{D_iu_sD_ku_s}{1+|Du_s|^2}\right)\dfrac{D_{kj}(u_2-u_1)}{\sqrt{1+|Du_s|^2}}\right]
\\
&\hspace{.5cm}\:+\dfrac{\pI{D(u_2-u_1),Du_s}}{(1+|Du_s|^2)^{\frac{3}{2}}}.
\end{align*}
On the other hand, by the Mean Value Theorem, there exists $s_0\in(0,1)$ such that, 
\begin{align}\label{0=}
0=\int\limits_{0}^1\dfrac{\partial E}{\partial s}(u_s)ds=\dfrac{\partial E}{\partial s}(u_{s_{0}}). 
\end{align}
We claim that the function $v=u_2-u_1$ is constant in $B_r(0)$. In fact, $v$ reaches an interior maximum at $0$. In particular, $D^2v$ is definite non-positive and $Dv=0$ at $0$. Then, by  Equation \eqref{0=} at $0$, we have
\begin{align}\label{1=}
    0=\dfrac{\partial Q_k}{\partial h^i_j}(h^i_j(s_0))\left(\delta_{ik}-\dfrac{D_iu_{s_{0}}D_ku_{s_{0}}}{1+|Du_{s_{0}}|^2}\right)\dfrac{D_{kj}v}{1+|Du_{s_{0}}|^2}.
\end{align}
Moreover, since the principal curvatures of $u_{s_{0}}$ lies in $\Gamma_n$, and by choosing  a smaller $r$ if it is necessary,  we may assume that the term 
\begin{align*}
\dfrac{\partial Q_k}{\partial h^i_j}(h^i_j(s_0))\left(\delta_{ik}-\dfrac{D_iu_{s_{0}}D_ku_{s_{0}}}{1+|Du_{s_{0}}|^2}\right)
\end{align*}
is uniformly positive in $B_r(0)$. Then, by Hopf's Maximum Principle, $v\equiv 0$ in $B_r(0)$. This means that $M_1\cap B(p,r)=M_2\cap B(p,r)$. Finally,  since both graphs are connected we may apply the Uniqueness Continuation Principle for uniformly elliptic linear operators, see for instance \cite{gilbra_trudinger}, to obtain that $M_1=M_2$. 

For the Boundary Tangency Principle, we only need to change $B(p,r)$ with a hemisphere $B(p,r)\cap \set{x_{n+1}\geq 0}$, here we consider $\Pi=\set{x_{n+1}=0}$. Then, since the intersection $M_i\cap\Pi$ is transversal, the function $v=u_2-u_1$ satisfies $\dfrac{\partial v}{\partial N}(p)=0$ where $N=e_1$ is the normal unit vector of $\partial M_i$ at $p$. On the other hand, by the Boundary Hopf's Maximum Principle, $\dfrac{\partial v}{\partial N}(p)>0$ a contradiction. Therefore, the same argument holds  in this case to conclude the result.   
\end{proof}

We end this section with the proofs of the non-existence theorem for entire $Q_{n-1}$-translator in $\rr^{n+1}$. Recall that an entire hypersurface in $\rr^{n+1}$ corresponds to a graph of a function defined on all $\rr^n$.  
\begin{proof}[Proof of Theorem \ref{T3}]
	We prove this by contradiction. Assume that there is an entire $Q_{n-1}$-translator $M$ given by a function $u:\rr^n\to\rr$ such that its principal curvatures belong to the cone $\Gamma_{n}$. This last property implies that $M$ is strictly convex.
	\\
	On the other hand, by Theorem \ref{T1},  there is a \textbf{ros} $Q_{n-1}$-translator $C_{n-1}$ of the form \eqref{M}  for each $n\geq 2$. Recall that $C_{n-1}$ is a complete, non-entire and strictly convex hypersurface. 
	\\
	Then, translating suitably $C_{n-1}$ over $M$, we can find a $t_0>0$ such that $C_{n-1}+t\eps_{n+1}$ lies strictly above from $M$ for $t\geq t_0$. Note that this can be done since $C_{n-1}$ is not an entire graph. Now, we may translate $C_{n-1}+t\eps_{n+1}$ downward  until it touches $M$ for the first time. Finally, by Theorem \ref{T2}, we obtain that $M=C_{n-1}$ which is a contradiction.
\end{proof}

\section{Proof of Theorem \ref{T4}}\label{sec 5}

We use the Method of Moving Planes to show that $Q_{n-1}$-translators $M\subset\rr^{n+1}$ which are strictly convex and asymptotic to the cylinder $\Sp^{n-1}(r_n)$ are rotationally symmetric, and consequently the $Q_{n-1}$-translator found in Theorem \ref{T1}. 
\newline
In addition, we base our proof by the given in \cite{Paco_2014} for translating soliton of the Mean Curvature flow with one end asymptotic to the bowl soliton. 
\newline

We first mention that  $Q_k$-translators are invariant under rotation around the $x_{n+1}$-axis.  Indeed, let $R$ be a rotational field which fix the $x_{n+1}$-axis and $M=F(\Omega)$ be a $Q_k$-translator such that its principal curvatures belong to $\Gamma_{k}$. Then, it follows that the function 
\begin{align*}
f=\pI{R(F),\nu},
\end{align*}
satisfies the equation 
\begin{align*}
\square_k f+\nabla_{n+1}f+\abs{A_k}^2=0,
\end{align*}
where $|A_k|^2=\dfrac{\partial Q_k}{\partial h_{ij}}h_{il}h_{lj}$. This fact can be seen by taking the flow 
$\partial_t\varphi_t=R\circ\varphi_t$ where $\varphi_t$ is a 1-parameter of diffeomorphism such that $\varphi_0=Id$. Then, by composing these family with the $Q_k$-flow of $M$ (i.e:$F_t=F+te_{n+1}$), it follows
\begin{align*}
\partial_t(\varphi_t\circ F_t)=R\circ(\varphi_t\circ F_t).
\end{align*}
Finally, by choosing a normal reparametrization $\phi_t$ such that 
\begin{align*}
\partial_t(\varphi_t\circ F_t\circ\phi_t)=R\circ(\underbrace{\varphi_t\circ F_t\circ\phi_t}_{\tilde{F_t}})=f(\tilde{F}_t)\nu(\tilde{F}_t),
\end{align*}
and taking time derivatives, we can see that
\begin{align*}
\partial_tQ_k&=\dfrac{\partial Q_k}{\partial h_{ij}}\partial_th_{ij}=\dfrac{\partial Q_k}{\partial h_{ij}}\left(\nabla_i\nabla_j f+h_{il}h_{lj}\right)=\square_kf+|A_k|^2,
\\
\partial_tQ_k&=\partial_t\pI{\nu,e_{n+1}}=-\pI{\nabla f,e_{n+1} }=-\nabla_{n+1}f. 
\end{align*}
In consequence, under the hypothesis of Theorem \ref{T4}, it is enough to show that $M$ is symmetric along the plane $\set{x_1=0}$.  
\newline
To this purpose, we set the following notation.  Let $M\subset\rr^{n+1}$ be strictly convex  $Q_{n-1}$-translator asymptotic to the cylinder $\Sp^{n-1}(r_n)\times \rr$.  We set,
\begin{align*}
&\Pi(t)=\set{x_1=t}, M_+(t)=M\cap\set{x_1\geq t},M_{-}(t)=M\cap\set{x_1\leq t},
\\
&M_{+}^{*}(t)=\set{(2t-x_1,x_2,\ldots,x_{n+1})\in\rr^{n+1}:(x_1,\ldots,x_{n+1})\in M_+(t)}, 
\\
&M_{-}^{*}(t)=\set{(2t-x_1,x_2,\ldots,x_{n+1})\in\rr^{n+1}:(x_1,\ldots,x_{n+1})\in M_{-}(t)},
\end{align*}
recall that $M_{\pm}^*(t)$ is just the reflection of $M_{\pm}(t)$ along the plane $\Pi(t)$ (see Fig 2.).
\newline

\begin{figure} \label{fig3}
	\centering
	\includegraphics[width=0.33\columnwidth]{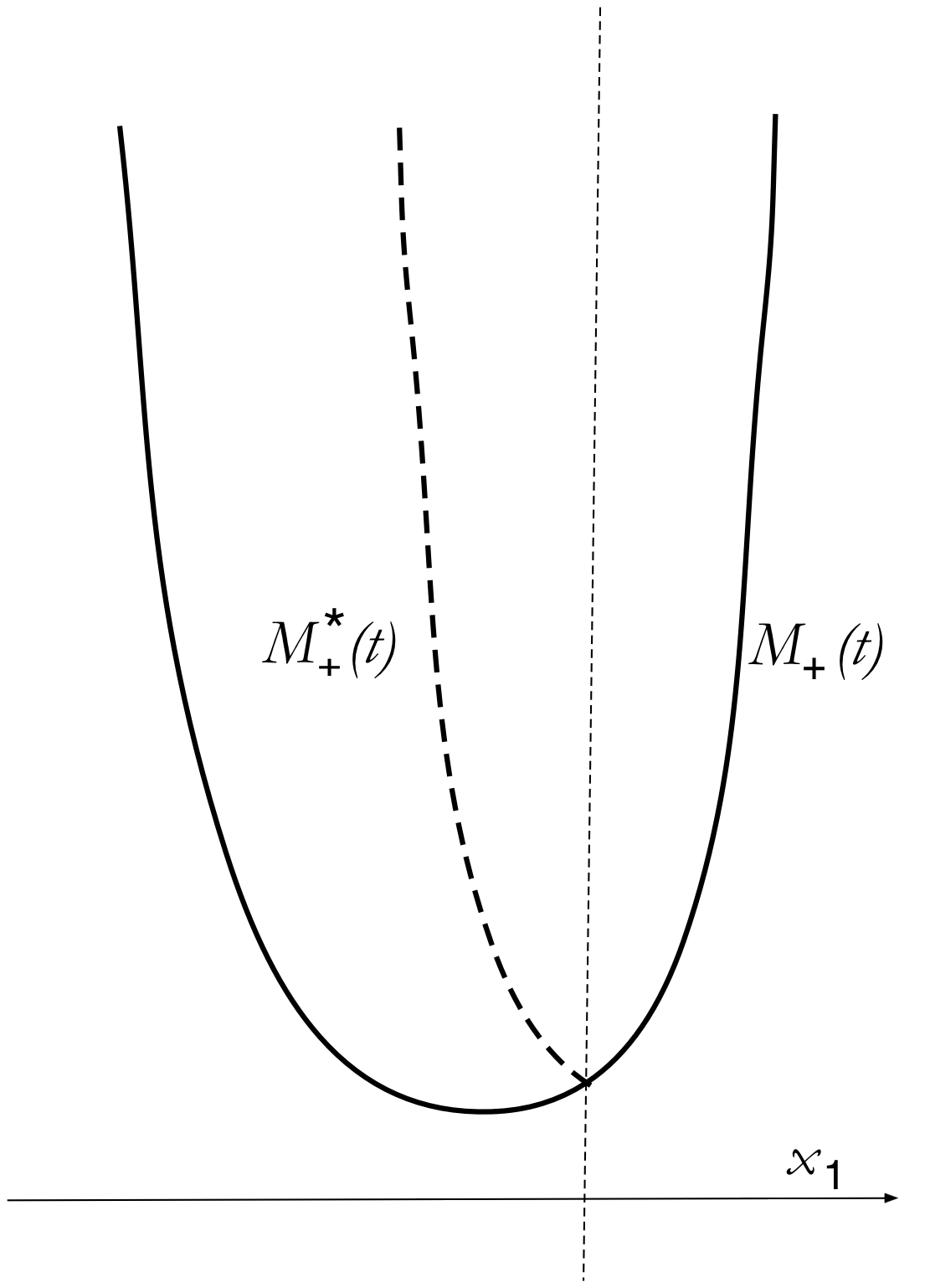}
	\vskip-1mm
	\caption{Alexandrov's Moving planes methods.}
\end{figure}

\begin{definition}\label{Def Order}
	Let $A,B$ be two set of $\rr^{n+1}$. We say that $A$ is on the right side of $B$, denoted by $B\leq A$ if, and only if, for every $x\in \Pi(0)$ such that,
	\begin{align*}
\pi_1^{-1}(\set{x})\cap A\neq \emptyset \mbox{ and }	\pi_1^{-1}(\set{x})\cap B \neq\emptyset,
	\end{align*}
where $\pi_1(x_1,\ldots,x_{n+1})=(0,x_2,\ldots,x_{n+1})$, it follows that 
\begin{align*}	
\sup\set{p_1(p): p\in \pi_1^{-1}(\set{x})\cap B}\leq \inf\set{p_1(p): p\in\pi_1^{-1}(\set{x})\cap A},
\end{align*}
here $p_1(x_1,\ldots,x_{n+1})=x_1$. 
\end{definition}
Recall that for arbitrary sets, the relation $B\leq A$ is not a partial order, but for sets given by the graph of a function over planes $\Pi(0)$ works as a partial order. 
\newline

Consequently, the moving plane method consists in to show
that the set 
\begin{align*}
I=\set{t\in [0,r_n):M_+(t) \mbox{ is a graph over }\Pi(0) \mbox{ and } M_-(t)\leq M_+^*(t)},
\end{align*}
is the interval $[0,r_n)$. Then, by similar symmetric arguments, we finally obtain $M_+^*(0)=M_-(0)$ giving that $M$ is symmetric along the $\Pi(0)$ plane. To this purpose, we show that $I$ is a not empty, open and close set of $[0,r_n)$ whose minimum is $0$. 
\begin{claim}\label{Claim1}
The set $I\neq\emptyset$.
\end{claim}
\begin{proof}
	 First, since $M$ is asymptotic to the cylinder  $\Sp^{n-1}(r_n)\times\rr$, we can choose $0<r_\delta<r_n$ such that the coordinates of $M\setminus B^{n+1}(0,r_\delta)$ satisfy $r_n-\delta\leq|x_i|< r_n$ for $i=1,\ldots,n$ and $r_\delta<|x_{n+1}|$. Here $B^{n+1}(0,r_\delta)$ denotes an Euclidean ball of $\rr^{n+1}$.

	 Let $t\in [r_n-\delta,r_n)$. Then, by the strict convexity of $M$, the Gauss map $\nu$ of $M$ maps $M\setminus B^{n+1}(0,r_\delta)$ to a closed region whose boundary is the equator of $\Sp^{n-1}$. In particular, the euclidean distance satisfies
	 \begin{align*}
	 \dist{\nu(M_+(t)), \Pi(0)}>0,
	 \end{align*}
	 which means that $M_+(t)$ can be written by the graph of a function $u_{t}=x_1-t$ such that $0 \leq u_t<\delta$ over a region $\Omega_t$, here $x_1=p_1(x)$ for $x\in M_+(t)$.
	 
	 Second, for the other property, we note that $M_-\left(\dfrac{r_n-\delta}{2}\right)\leq M_+^*\left(\dfrac{r_n-\delta}{2}\right)$. Indeed, according to Definition \ref{Def Order}, we have
	 \begin{align*}
	r_n-\delta -x_1-y_1\geq\dfrac{r_n-\delta}{2}-y_1\geq0,
	 \end{align*} 
	 where $y_1=p_1(y)$ and $y\in M_-(t)$. Therefore, $\dfrac{r_n-\delta}{2}\in I$.
	 
	 \end{proof}

\begin{claim}
	For every $t\in I$, we have $[t,r_n)\subset I$. In particular, $I$ is an open set. 
\end{claim}	

\begin{proof}
	Let $t\in I$ and $t<s<r_n$. Then, by definition, $M_+(t)$ can be written by the graph of the function $u_t=x_1-t\geq 0$ defined over a region $\Omega_t$, where $x_1=p_1(x)$ and $x\in M_+(t)$.  First, we note that $\partial\Omega_t=M\cap\Pi(t)$. This fact, together with the equation 
	\begin{align*}
	\square_{n-1}u_t=-Q_{n-1}\pI{\nu,e_1}<0,
	\end{align*}
	implies that $u_t$ cannot reach a local minimum in the interior of $\Omega_t$, since otherwise, $u_t$ will be constant, which contradicts that $M$ is strictly convex and asymptotic to the cylinder $\Sp^{n-1}(r_n)\times\rr$. Therefore, $\Omega_t$ does not contain any compact connected components, and by the asymptotic behavior of $M$, must be a closed unbounded set of $\Pi(t)$.   
	\newline
	 Moreover, by definition, we have that $M_+(s)$ can be written by the graph of a function defined over the plane $\Pi(t)$, since $M_+(s)\subset M_+(t)$. Thus, $M_+(s)$ is graph over $\Pi(0)$. In addition, we note that 
	 \begin{align*}
	\inf\set{ p_{n+1}(x):x\in M\cap\Pi(t)}\leq \inf\set{ p_{n+1}(x):x\in M\cap\Pi(s)},
	\end{align*}
	where $p_{n+1}(x_1,\ldots,x_{n+1})=x_{n+1}$. This fact implies, $\Omega_s\subset\Omega_t$. Then, since $M_-(t)\leq M_+^*(t)$, it follows that $M_-(s)\leq M_+^*(s)$. Consequently, $[s,r_n)\subset I$. 
\end{proof}

\begin{claim}
	The set $I$ is closed.
\end{claim}

\begin{proof}
Let $t_n\in I$ a sequence of point such that $t_n\to t_0$. By previous part we have that $(t_0,r_n)\subset I$. First, we show by contradiction that $M_+(t_0)$ is a graph of the function $u_{t_{0}}=x_1-t_0$ over a region $\Omega_{t_{0}}\subset \Pi(t_0)$.
\newline
Let $p,q\in M_+(t_0)$ such that $p_i(p)=p_i(q)$ for $i>2$ and $p_1(p)<p_1(q)$. Then, by the previous part, it follows that $p_1(p)=t_0$. Consequently, we have two possibilities for $p$, it belongs in the interior of $\Omega_{t_{0}}$ or at the boundary $\partial \Omega_{t_{0}}=M\cap \Pi(t_0)$.
\newline 
Note that for the first case, we can choose $\eps>0$ such that 
\begin{align*}
p_1(q)>p_1(p)+2\eps=t_0+2\eps.
\end{align*} 
 Then, it follows that
 \begin{align*}
2(t_0+\eps)-p_1(q)<t_0=p_1(p), 
\end{align*}
a contradiction with $M_-(t_0+\eps)\leq M_+^*(t_0+\eps)$ for every $\eps>0$.
\newline
For the second case, we consider the cylindrical solid $C=\set{(t,z):t\in[t_0,r_n),\:z\in\Omega_{t_{0}} }$. Then, $C$ intersect $M_+(t_0)$ at points from the outside of $\Pi(t_0)$. But this contradicts the fact of $M_+(t_0+\eps)$ being a graph over $\Pi(t_0)$ for sufficiently small $\eps>0$.  Therefore, $M_+(t_0)$ is the graph of the function $u_{t_0}=x_1-t_0$ over the region $\Omega_{t_{0}}\subset \Pi(t_0)$. Finally, we note that the relation $M_-(t_0)\leq M_+^*(t_0)$ will follows by the continuity of the family $u_{t_n}=x_1-t_n$. 
	\end{proof}

\begin{claim}
	The minimum of $I$ is $0$.
\end{claim}

\begin{proof}
We argue by contradiction, we assume that $\inf I=t_0>0$. As we did in Claim \ref{Claim1}, there exist some positive number $h>r_\delta$ such that the Euclidean distance
\begin{align}\label{Fact}
\dist{\nu \left(M_+(t_0)\cap \set{x_{n+1}>h}\right), \Pi(0)}>0.
\end{align}
Therefore, we can choose $\eps_0$ such that $M_+(t_0-\eps_0)\cap \set{x_{n+1}>h}$ can be written by the graph of a function over $\Pi(0)$, and satisfies 
\begin{align}\label{Eq 5.5}
M_-(t_0-\eps)\cap\set{x_{n+1}>h}\leq M_+^*(t_0-\eps_0)\cap\set{x_{n+1}>h}. 
\end{align}
Now we consider the compact set
\begin{align*}
K=M\cap\set{x_{n+1}\leq h}.
\end{align*}
Since $t_0\in I$, it follows that $K_+(t_0)$ is a graph over the hyperplane $\Pi(0)$. Moreover, by the strict convexity of $M$ it follows that 
\begin{align*}
\nu\left(K_+(t_0)\right)\cap\Pi(0)=\emptyset,
\end{align*}
\begin{figure} \label{fig2}
	\includegraphics[width=0.33\columnwidth]{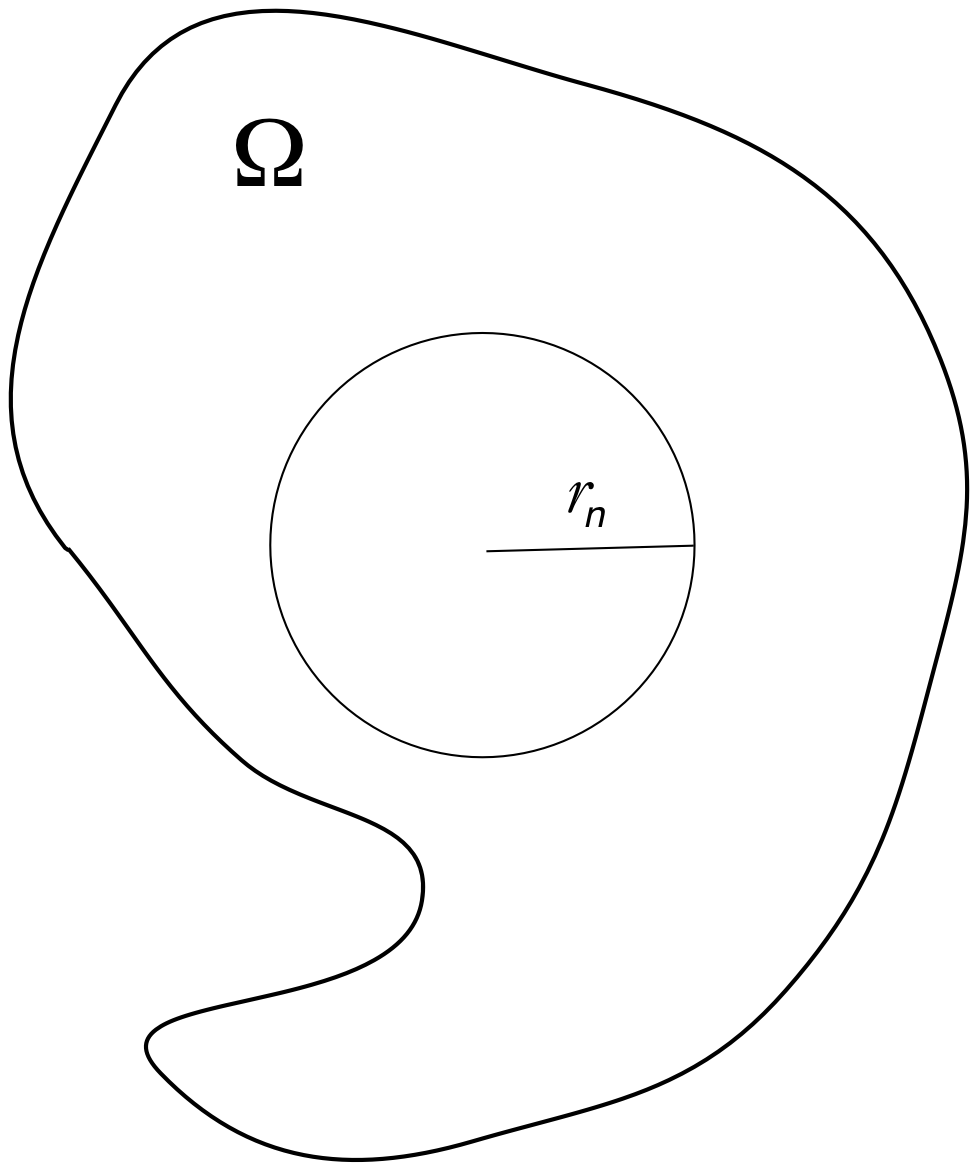}\hspace{3cm} \includegraphics[width=0.33\columnwidth]{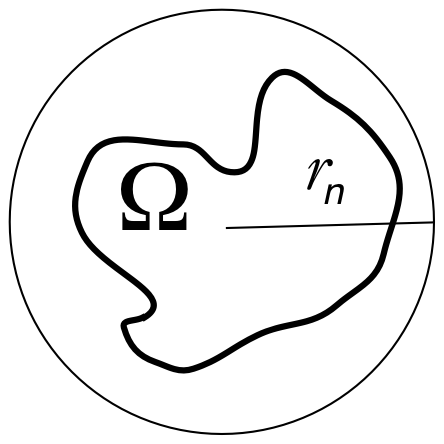}
	\vskip-1mm
	\caption{Non-existence of $Q_{n-1}$-translator.}
\end{figure}
because otherwise we would have a vanishing principal curvature of $M$. Furthermore, by compactness of $K_+(t_0)$, there exist $\eps_1\in (0,\eps_0)$ such that 
\begin{align*}
K_+(t)\cap\Pi(0)=\emptyset,\: t\in[t_0-\eps_1, t_0].
\end{align*}
In particular, this means that $K_+(t)$ is a graph over $\Pi(0)$ for $t\in[t_0-\eps_1,t_0]$. This fact together with \eqref{Fact}, implies that $M_+(t)$ is a graph over $\Pi(0)$ for $t_0-\eps_1\leq t$.

 We note that 	
 	\begin{align*}
 	M_+^*(t)\cap M_-(t)\cap K\subset K_+(t_0-\eps_1),
 	\end{align*}
 	for $t\geq t_0-\eps_1$. In addition, since $M_-(t_0)\leq M_+^*(t_0)$, it follows that 
 	\begin{align*}
 	M_+(t_0)\cap M_-(t_0)=M\cap\Pi(t_0). 
 	\end{align*}
 	Indeed, if it not the case, let $p\in \left(M_+(t_0)\cap M_-(t_0)\right)\setminus\left(M\cap \Pi(t_0)\right)$.  Then, by the Interior Tangency principle, we have $M_+^*(t_0)=M_-(t_0)$. This means that the end of $M$ will have at $x_1=t_0>0$ a plane of symmetry, but this fact contradicts our asymptotic behavior of $M$.
 	\newline
 	Hence, by compactness, we can choose  $0<\eps_2\leq \eps_1$ such that  
 	\begin{align*}
 	M_+^*(t)\cap M_-(t)\cap K=\Pi(t)\cap K, \mbox{ for all } t\in [t_0-\eps_2,t_0].
 	\end{align*}
 	Then, it follows that,
 	\begin{align*}
  \left(M_+^*(t)\cap M_-(t)\right)\setminus M\cap\Pi(t)\subset M\cap\set{x_{n+1}>h}, \mbox{ for }t\in[t_0-\eps_2,t_0].  
   \end{align*}
  Thus, combining the above line with Equation \eqref{Eq 5.5}, we obtain 
  \begin{align*}
 M_+^*(t)\cap M_-(t)=M\cap \Pi(t),\mbox{ for }t\in[t_0-\eps_2,t_0].   
	\end{align*}     	
Then, by standard continuity arguments, it follows that $M_-(t)\leq M_+^*(t)$ for $t\in[t_0-\eps_2,t_0]$. Consequently, $t_0-\eps\in I$ a contradiction with $t_0=\inf I$. Therefore, $t_0=0$.
\end{proof} 
In particular, $M_-(0)\leq M_+^*(0)$. A symmetric argument also will show that $M_+(0)\leq M_-^*(0)$. Therefore, $M_+^*(0)=M_-(0)$, this shows that $M$ is symmetric with respect the plane $\Pi(0)$, finalizing the proof of Theorem \ref{T4}. 
 
\begin{remark}
	From theorems \ref{T2} and \ref{T4}, we can see that if $\Sigma$ is a $Q_{n-1}$-translator such that is a graph over a precompact set $\Omega\subset\rr^n$ with diameter $\mbox{diam}(\Omega)<2r_{n}$, it does not exist. In addition, the same result will follows if $\Omega$ contains a point such that $B(x,r_n)\subset \Omega$ (see Fig. 3).   
\end{remark}

%

\end{document}